\documentclass[a4, 12pt]{amsart}
\usepackage{amssymb}
\usepackage{amstext}
\usepackage{amsmath}
\usepackage{amscd}
\usepackage{latexsym}
\usepackage{amsfonts}
\usepackage{color}
\usepackage{enumerate}
\usepackage{textcomp}
\usepackage[all]{xy}

\usepackage{appendix}

\theoremstyle{plain}
\newtheorem{thm}{Theorem}[section]
\newtheorem*{thm*}{Theorem}
\newtheorem*{cor*}{Corollary}
\newtheorem*{defn*}{Definition}

\newtheorem{prop}[thm]{Proposition}
\newtheorem{lem}[thm]{Lemma}
\newtheorem{cor}[thm]{Corollary}
\newtheorem{claim}[thm]{Claim}
\newtheorem*{claim*}{Claim}
\newtheorem*{ac}{Acknowledgments}

\theoremstyle{definition}

\newtheorem{rem}[thm]{Remark}

\theoremstyle{remark}

\numberwithin{equation}{thm}
\def\m{\mathfrak m}

\def\p{\mathfrak p}
\def\q{\mathfrak q}

\newcommand{\calD}{\mathcal{D}}

\newcommand{\calS}{\mathcal{S}}

\def\depth{\mathrm{depth}}
\def\Supp{\mathrm{Supp}}

\def\Ass{\mathrm{Ass}}

\begin{document}

\setlength{\baselineskip}{12pt}
%%%%%%%%%%%%%%%%%%%%%%%%%%%%%%%%%%%%%%%%%%%%%%%%%%%%%%%%%%%%%
\title{Topics on sequentially Cohen-Macaulay modules}
\author{Naoki Taniguchi}
\address{Department of Mathematics, School of Science and Technology, Meiji University, 1-1-1 Higashi-mita, Tama-ku, Kawasaki 214-8571, Japan}
\email{taniguti@math.meiji.ac.jp}
\urladdr{http://www.isc.meiji.ac.jp/~taniguci/}

\author{Tran Thi Phuong}
\address{Department of Mathematics, School of Science and Technology, Meiji University, 1-1-1 Higashi-mita, Tama-ku, Kawasaki 214-8571, Japan}
\email{sugarphuong@gmail.com}

\author{Nguyen Thi Dung}
\address{Thai Nguyen University of Agriculture and Forestry, Thai Nguyen, Vietnam}
\email{xsdung050764@gmail.com}

\author{Tran Nguyen An}
\address{Thai Nguyen University of Pedagogical, Thai Nguyen, Vietnam}
\email{antrannguyen@gmail.com}


\thanks{The first author was partially supported by Grant-in-Aid for JSPS Fellows 26-126 and by JSPS Research Fellow. The second author was partially supported by JSPS KAKENHI 26400054. The third and the fourth author were partially supported by a Grant of Vietnam Institute for Advanced Study in Mathematics (VIASM) and Vietnam National Foundation for Science and Technology Development (NAFOSTED). }

\begin{abstract}
In this paper, we study the two different topics related to sequentially Cohen-Macaulay modules. The questions are when the sequentially Cohen-Macaulay property preserve the localization and the module-finite extension of rings.
\end{abstract}

\maketitle
\tableofcontents
\section{Introduction}

Throughout this paper, unless other specified, let $R$ be a commutative Noetherian ring, $M \neq (0)$ a finitely generated $R$-module of dimension $d$. Then there exists the largest $R$-submodule $M_n$ of $M$ with $\dim_R M_n \le n$ for every $n \in \Bbb Z$ (here $\dim_R(0) = -\infty$ for convention).
Let 
\begin{eqnarray*}
\calS (M) &=& \{\dim_R N \mid N \ \text{is~an~}R\text{-submodule~of~}M, N \ne (0)\} \\
                 &=& \{\dim R/\p \mid \p \in \Ass_RM \}.
\end{eqnarray*}
Put $\ell = \sharp \calS (M)$ and write $\calS (M) = \{d_1< d_2 < \cdots < d_\ell = d\}$. Then {\it the dimension filtration of $M$} is a chain
$$
\calD : D_0 :=(0) \subsetneq D_1 \subsetneq  D_2 \subsetneq  \ldots \subsetneq D_{\ell} = M
$$
of $R$-submodules of $M$, where $D_i = M_{d_i}$ for every $1 \le i \le \ell$. 
Notice that the notion of dimension filtration is due to S. Goto, Y. Horiuchi and H. Sakurai (\cite{GHS}) and a little different from that of the original one given by P. Schenzel (\cite{Sch}). However, let us adopt the above definition throughout this paper. 
We say that $M$ is {\it a sequentially Cohen-Macaulay $R$-module}, if the quotient module $C_i = D_i/D_{i-1}$ of $D_i$ is Cohen-Macaulay for every $1 \le i \le \ell$. In particular, the ring $R$ is called {\it a sequentially Cohen-Macaulay ring}, if $\dim R < \infty$ and $R$ is a sequentially Cohen-Macaulay module over itself.

The aim of this paper is to investigate the two following questions. The first one is that does the sequentially Cohen-Macaulay property is inherited from localizations, which has been already studied by Cuong-Goto-Truong (\cite{CGT}) and Cuong-Nhan (\cite{CN}) in the case of local rings. 
In Section 2, we shall probe into a possible generalization of their results (\cite[Proposition 2.6]{CGT}, \cite[Proposition 4.7]{CN}). More precisely, we prove the following.

\begin{thm}\label{4.1}
Suppose that $\dim R/\p = \dim R_P/{\p R_P}$ for every $\p \in \Ass_R M$ and for all  maximal ideal $P$ of $R$ such that $\p \subseteq P$.
Then the following conditions are equivalent.
\begin{enumerate}[$(1)$]
\item $M$ is a sequentially Cohen-Macaulay $R$-module. 
\item $M_P$ is a sequentially Cohen-Macaulay $R_P$-module for every $P \in \Supp_R M$.
\end{enumerate}
\end{thm}

Another one is the question of whether the sequentially Cohen-Macaulay property preserve the module-finite extension of rings or not. In Section 3, especially, we will give the characterization of sequentially Cohen-Macaulay local rings obtained by the idealization (that is, trivial extension), which is stated as follows.

\begin{thm}[Corollary \ref{cor5}] 
Suppose that $R$ is a local ring. Let $R\ltimes M$ denote the idealization of $M$ over $R$. Then the following conditions are equivalent.
\begin{enumerate}
\item[$(1)$] $R \ltimes M$ is a sequentially Cohen-Macaulay local ring.
\item[$(2)$] $R\ltimes M$ is a sequentially Cohen-Macaulay $R$-module. 
\item[$(3)$] $R$ is a sequentially Cohen-Macaulay local ring and $M$ is a sequentially Cohen-Macaulay $R$-module. 
\end{enumerate}
\end{thm}

%%%%%%%%%%%%%%%%%%%%%%%%%%%%%%%%%%%%%%%%%%%%%%%%%%%%%%%%%%%%%%%%%%%%%%%%%%%%%%%%%%%%%%%%%%%%%%%%%%%%%%%%%%%%%%%%%%%%%%%%%%%%%%%%%%%%%%%%%%%%%%%%%%%%%%%%%%%%%%%%%%%%%%%%%%%%%%%%%%%%%%%%%%%%%%%%%%%%%%%%%%%%%%%%%%%%%%%%%%%%%%%%%%%%%%%%%%%%%%%

\section{Localization of sequentially Cohen-Macaulay modules}

The purpose of this section is mainly to prove Theorem \ref{4.1}.

\begin{proof}[Proof of Theorem \ref{4.1}]
$(1) \Rightarrow (2)$~~We may assume that $\ell > 1$ and the assertion holds for $\ell -1$. Thanks to \cite[Proposition 2.6]{CGT}, it is enough to show the case where $P$ is a maximal ideal of $R$. Then we get the exact sequence $0 \to N \to M \to C \to 0$ of $R$-modules where $N = D_{\ell -1}$ and $C = C_{\ell}$. We may also assume $N_P \neq (0)$, $C_P \neq (0)$ and $\dim_{R_P}N_P = \dim_{R_P}M_P$, since $N_P$ is sequentially Cohen-Macaulay and $C_P$ is Cohen-Macaulay.
By using the hypothesis, we get $\dim_{R_P}M_P = \dim_{R_P}C_P$ (see also \cite[Corollary 2.3]{Sch}).

Let
$$
E_0 = (0) \subsetneq E_1 \subsetneq \cdots \subsetneq E_{t-1} \subsetneq E_t = N_P
$$
be the dimension filtration of $N_P$. Then $M_P/E_{t-1}$ is a Cohen-Macaulay $R_P$-module of dimension $\dim_{R_P}M_P$, because $N_P/E_{t-1}$ and $M_P/N_P$ are Cohen-Macaulay $R_P$-modules of same dimension $\dim_{R_P}M_P$. Hence $M_P$ is a sequentially Cohen-Macaulay $R_P$-module.

$(2) \Rightarrow (1)$~~
Suppose that $C_i$ is not Cohen-Macaulay for some $1 \le i \le \ell$. Then there exists a maximal ideal $P$ of $R$ such that $[C_i]_P \neq (0)$, $[C_i]_P$ is not a Cohen-Macaulay $R_P$-module. Let $\alpha = \dim_{R_P}[C_i]_P$. We choose $\p \in \Ass_R C_i$ such that $\p \subseteq P$, $\alpha = \dim R_P/{\p R_P}$. Then 
$$
\alpha = \dim R_P/{\p R_P} = \dim R/{\p} = d_i
$$
by using the hypothesis and $\p \in \Ass_R C_i$. Therefore  $d_i \in \calS(M_P) \subseteq \calS(M)$, since 
$$
\calS(M_P) = \{\dim R/{\p} \mid \p \in \Ass_R M, ~\p \subseteq P\} \subseteq \calS(M).
$$

Let $q = {\sharp}\calS(M_P)$. We write $\calS(M_P) = \{n_1 < n_2 < \cdots < n_q\}$. Then $d_i = n_j$ for some $1 \le j \le q$. Let $(0) = \bigcap_{\p \in \Ass_R M}M(\p)$ be a primary decomposition of $(0)$ in $M$. In this case
$$
(0) = \bigcap_{\p \in \Ass_R M, ~ \p \subseteq P}[M(\p)]_P
$$
 forms a primary decomposition of $(0)$ in $M_P$.
 
\begin{claim}\label{4.3}
The following assertions hold true.
\begin{enumerate}[$(1)$]
\item $[D_i]_P = D_j(M_P)$
\item $[D_{i-1}]_P = D_{j-1} (M_P)$
\end{enumerate}
where $\{D_j (M_P)\}_{0 \le j \le q}$ stands for the dimension filtration of $M_P$.
\end{claim}
\begin{proof}[Proof of Claim \ref{4.3}]
$(1)$~~We may assume that $i < \ell$. Then 
$$
D_i = \bigcap_{\p \in \Ass_R M, ~\dim R/{\p} > d_i}M(\p),
$$
so that
$$
[D_i]_P = \bigcap_{\p \in \Ass_R M, ~\dim R/{\p} > n_j, ~ \p \subseteq P}[M(\p)]_P.
$$

We now assume that $\p \nsubseteq P$ for every $\p \in \Ass_R M$ with $\dim R/{\p} > d_i$. Then $M(\p)_P = M_P$, so that $[D_i]_P = M_P$. Then $\dim_{R_P}M_P = d_i$, since $\dim_{R_P}M_P \le d_i \in \calS(M_P)$. Therefore $d_i = n_q$, $j=q$ and 
$$
[D_i]_P = M_P = D_q(M_P) = D_j (M_P).
$$
Thus we may assume that $\p \subseteq P$ for some $\p \in \Ass_R M$ with $\dim R/{\p} > d_i$. Thus
$$
n_q = \dim_{R_P}M_P \ge \dim_{R_P}R_P/{\p R_P} = \dim R/{\p} > d_i = n_j
$$
whence $1 \le j < q$. Hence 
$$
[D_i]_P = \bigcap_{\p \in \Ass_R M, ~\dim R_P/{\p R_P} > n_j, ~ \p \subseteq P}[M(\p)]_P = D_j (M_P).
$$

$(2)$~~We get 
$$
[D_{i-1}]_P = \bigcap_{\p \in \Ass_R M, ~ \dim R/{\p} \ge d_i, ~\p \subseteq P}[M(\p)]_P,
$$
since $D_{i-1} = \bigcap_{\p \in \Ass_R M,~\dim R/{\p} \ge d_i}M(\p)$. We may assume $j>1$. Then $d_i = n_j > n_{j-1}$, so that 
\begin{eqnarray*}
[D_{i-1}]_P &=& \bigcap_{\p \in \Ass_R M,~\dim R/{\p} \ge d_i,~\p \subseteq P}[M(\p)]_P = \bigcap_{\p \in \Ass_R M,~\dim R_P/{\p R_P}> n_{j-1}}[M(\p)]_P  \\
                  &=& D_{j-1}(M_P).
\end{eqnarray*}
\end{proof}
 Hence $C_j (M_P) = D_j(M_P)/{D_{j-1}(M_P)}$ is Cohen-Macaulay, since $M_P$ is sequentially Cohen-Macaulay, whence $[C_i]_P$ is a Cohen-Macaulay $R_P$-module, a contradiction.
 \end{proof}

If the base ring $R$ is a finitely generated algebra over a field, then the assumption of Theorem \ref{4.1} is automatically satisfied and we get the following.

\begin{cor}\label{4.4}
Let $R$ be a finitely generated algebra over a field, $M \neq (0)$ a finitely generated $R$-module. Then the following conditions are equivalent.
\begin{enumerate}[$(1)$]
\item $M$ is a sequentially Cohen-Macaulay $R$-module. 
\item $M_P$ is a sequentially Cohen-Macaulay $R_P$-module for every $P \in \Supp_R M$.
\end{enumerate}
\end{cor}

From now on, let us explore the localization property for the graded rings. Let $R=\sum_{n \in \Bbb Z}R_n$ be a Noetherian $\Bbb Z$-graded ring such that $R$ is an $H$-local ring with an $H$-maximal ideal $P$ of $R$. Let $M \neq (0)$ be a finitely generated graded $R$-module of dimension $d$. 
Let $\{D_i\}_{0 \le i \le \ell}$ be the dimension filtration of $M$. We put $q = \dim R/{P}$.
For an arbitrary ideal $I$ of a graded ring, $I^*$ stands for the ideal generated by every homogeneous elements in $I$.

\begin{lem}\label{4.5}
$\dim_R M = \dim_{R_P}M_P + q$. Therefore $\dim_R M = \dim_{R_{\m}}M_{\m}$  for every maximal ideal $\m$ of $R$ with $\m \supseteq P$. 
\end{lem}

\begin{proof}
We may assume $q > 0$ (thus $q = 1$).  Let $\m$ be a maximal ideal of $R$ such that $\m \supseteq P$ and $\dim R_{\m}/P R_{\m} = 1$. Let $\p \in \Ass_R M$ such that $P \supseteq \p$ and $\dim R_P/\p R_P = \dim_{R_P}M_P$. Then we have $\dim_R M \ge \dim_{R_P}M_P + 1$. Conversely, we choose $\p \in \Ass_R M$ and a maximal ideal $\m$ of $R$ such that $\p \subseteq \m$, $\dim R_{\m}/\p R_{\m} = d$. Since $\p$ is a graded ideal of $R$, $\p \subseteq \m^*$. Notice that $\m$ is not a graded ideal of $R$ because $q=1$. Hence we get $\m^* \subseteq P$ and
$$
d = \dim R_{\m}/\p R_{\m} = \dim R_{\m^*}/\p R_{\m^*} +1 \le \dim  R_P/ {\p R_P} + 1 \le \dim_{R_P}M_P + 1. 
$$
\end{proof}

Apply \cite[Theorem 2.3]{GHS} and Lemma \ref{4.5}, we get the following.

\begin{cor}\label{4.6}
The following assertions hold true.
\begin{enumerate}[$(1)$]
\item $[D_0]_\m = (0) \subsetneq [D_1]_\m \subsetneq \ldots \subsetneq [D_{\ell}]_\m = M_\m$ is the dimension filtration of $M_\m$ for every maximal ideal $\m$ of $R$ such that $\m \supseteq P$.
\item $[D_0]_P = (0) \subsetneq [D_1]_P \subsetneq \ldots \subsetneq [D_{\ell}]_P = M_P$ is the dimension filtration of $M_P$, so that $M$ is a sequentially Cohen-Macaulay $R$-module if and only if $M_P$ is a sequentially Cohen-Macaulay $R_P$-module.
\end{enumerate}
\end{cor}

\if0
\begin{rem}\label{4.7}
Suppose that $P$ is a maximal ideal of $R$. Then $[D_0]_P = (0) \subsetneq [D_1]_P \subsetneq \ldots \subsetneq [D_{\ell}]_P = M_P$ is the dimension filtration of $M_P$, so that $M$ is a sequentially Cohen-Macaulay $R$-module if and only if $M_P$ is a sequentially Cohen-Macaulay $R_P$-module.
\end{rem}
\fi

Finally we reach the goal of this section.

\begin{thm}\label{4.8}
The following conditions are equivalent.
\begin{enumerate}[$(1)$]
\item $M$ is a sequentially Cohen-Macaulay $R$-module.
\item $M_P$ is a sequentially Cohen-Macaulay $R_P$-module.
\end{enumerate}
When this is the case, $M_\p$ is a sequentially Cohen-Macaulay $R_{\p}$-module for every $\p \in \Supp_R M$.
\end{thm}

\begin{proof}
The equivalence of conditions $(1)$ and $(2)$ follows from Corollary \ref{4.6}. Let us make sure of the last assertion. Note that $\p^* \subseteq P$ for any $\p \in \Supp_R M$. Thanks to \cite[Proposition 2.6]{CGT}, $M_{\p^*}$ is a sequentially Cohen-Macaulay $R_{\p^*}$-module. Since $\p^* R_{(\p)}$ is an $H$-maximal ideal of the homogeneous localization $R_{(\p)}$ of $R$, $M_{(\p)}$ is a sequentially Cohen-Macaulay $R_{(\p)}$-module. We may assume that $\p$ is not a graded ideal of $R$, so that $\p R_{(\p)}$ is a maximal ideal of $R_{(\p)}$. Therefore $M_\p$ is a sequentially Cohen-Macaulay $R_\p$-module.
\end{proof}

\if0
\begin{rem}\label{8.8}
Let $R$ be a Noetherian ring, $M \neq (0)$ a finitely generated $R$-module with $\dim_R M < \infty$. Let $R[t]$ be a polynomial ring over $R$. We set $S =R[t]$ or $R[t, t^{-1}]$. Then the following conditions are equivalent.
\begin{enumerate}[$(1)$]
\item $M$ is a sequentially Cohen-Macaulay $R$-module.
\item $S\otimes_R M$ is a sequentially Cohen-Macaulay $S$-module.
\end{enumerate} 
\end{rem}
\fi

%%%%%%%%%%%%%%%%%%%%%%%%%%%%%%%%%%%%%%%%%%%%%%%%%%%%%%%%%%%%%%%%%%%%%%%%%%%

\section{Module-finite extension of sequentially Cohen-Macaulay modules}

In this section, we assume that $R$ is a local ring with maximal ideal $\m$. Remember that $M_n$ stands for the largest $R$-submodule of $M$ with $\dim_RM_n \le n$ for each $n \in \Bbb Z$. 

We note the following.

\begin{lem}\label{lemma1}
Let $M$ and $N$ be finitely generated $R$-modules. Then $[M \oplus N]_n = M_n \oplus N_n$ for every $n \in \Bbb Z$.
\end{lem}

\begin{proof}
We have $[M \oplus N]_n \supseteq M_n \oplus N_n$, since $\dim_R(M_n\oplus N_n) = \max\{\dim_RM_n, \dim_RN_n\} \le n$. Let $p: L = M\oplus N \to M, (x,y) \mapsto x$ be the first projection. Then $p(L_n) \subseteq M_n$, since $\dim_Rp(L_n) \le \dim_RL_n \le n$. We similarly have $q(L_n) \subseteq N_n$, where $q: M\oplus N \to N, (x,y) \mapsto y$ denotes the second projection. Hence $[M\oplus N]_n \subseteq M_n \oplus N_n$ as claimed.
\end{proof}

The following proposition includes the result \cite[Proposition 4.5]{CN}.

\begin{prop}\label{prop1} Let $M$ and $N$ $(M, N \ne (0))$ be  finitely generated $R$-modules. Then $M\oplus N$ is a sequentially Cohen-Macaulay $R$-module if and only if both $M$ and $N$ are sequentially Cohen-Macaulay $R$-modules.
\end{prop}

\begin{proof} We set $L = M \oplus N$ and $\ell = \sharp \calS (L)$. Then $\calS (L) = \calS (M) \cup \calS (N)$, as $\Ass_RL = \Ass_RM \cup \Ass_RN$. Hence if $\ell = 1$, then $\calS (L) = \calS (M) = \calS (N)$ and $\dim_RL = \dim_RM = \dim_RN$. Therefore when $\ell = 1$, $L$ is a sequentially Cohen-Macaulay $R$-module if and only if $L$ is a Cohen-Macaulay $R$-module, and the second condition is equivalent to saying that the $R$-modules $M$ and $N$ are Cohen-Macaulay, that is $M$ and $N$ are sequentially Cohen-Macaulay $R$-modules. Suppose that $\ell > 1$ and that our assertion holds true for $\ell - 1$. Let 
$$
D_0=(0) \subsetneq D_1 \subsetneq  D_2 \subsetneq  \cdots \subsetneq D_{\ell} = L
$$
be the dimension filtration of $L = M \oplus N$, where $\calS (L) = \{d_1< d_2 < \cdots < d_\ell\}$. Then $\{D_i/D_1\}_{1 \le i \le \ell}$ is the dimension filtration of $L/D_1$ and hence $L$ is a sequentially Cohen-Macaulay $R$-module if and only if $D_1$ is a Cohen-Macaulay $R$-module and $L/D_1$ is a sequentially Cohen-Macaulay $R$-module. Because 
$$
D_1 = \left\{
\begin{array}{lc}
M_{d_1} \oplus (0) & (d_1 \in \calS (M) \setminus \calS (N)),\\
\vspace{0.5mm}\\
M_{d_1} \oplus N_{d_1} & (d_1 \in \calS (M) \cap \calS (N)),\\
\vspace{0.5mm}\\
(0) \oplus N_{d_1} & (d_1 \in \calS (N) \setminus \calS (M))
\end{array}
\right.
$$ by Lemma \ref{lemma1},  the hypothesis on $\ell$ readily shows  the second condition is equivalent to saying that the $R$-modules $M$ and $N$ are sequentially Cohen-Macaulay.
\end{proof}

In what follows, let $A$ be a Noetherian local ring and assume that $A$ is a module-finite algebra over $R$. 

The main result of this section is the following.

\begin{thm} \label{thm2}
Let $M\neq (0)$ be a finitely generated $A$-module. Then the following assertions hold true.
\begin{enumerate}
\item[$(1)$] $M_n$ is the largest $A$-submodule of $M$ with $\dim_AM_n \le n$ for every $n \in \Bbb Z$.
\item[$(2)$] The dimension filtration of $M$ as an $A$-module coincides with that of $M$ as an $R$-module.
\item[$(3)$] $M$ is a sequentially Cohen-Macaulay $A$-module if and only if $M$ is a sequentially Cohen-Macaulay $R$-module.
\end{enumerate}
\end{thm}

\begin{proof}
Let $n \in \Bbb Z$ and $X$ denote the largest $A$-submodule of $M$ with $\dim_AX \le n$. Then $X \subseteq M_n$, since $\dim_RX =\dim_AX \le n$. Let $Y = AM_n$. Then $\dim_AY \le n$. In fact, let $\p \in \Ass_RY$. Then since $[M_n]_\p \subseteq Y_\p = A_\p{\cdot} [M_n]_\p \subseteq M_\p$, we see $[M_n]_\p \ne (0)$, so that $\p \in \Supp_RM_n$. Hence $\dim R/\p \le \dim_RM_n \le n$. Thus $\dim_AY = \dim_RY \le n$, whence $M_n \subseteq Y \subseteq X$, which shows $X = M_n$. Therefore assertions (1) and  (2) follows. Since $\dim_AL = \dim_RL$ and $\depth_AL = \depth_RL$ for every finitely generated $A$-module $L$, we get assertion (3).
\end{proof}

We summarize some consequences.

\begin{cor}\label{thm3} $A$ is a sequentially Cohen-Macaulay local ring if and only if $A$ is a sequentially Cohen-Macaulay $R$-module.
\end{cor}

\begin{cor}\label{cor1}
Let $M$ be a finitely generated $A$-module. Suppose that $R$ is a direct summand of $M$ as an $R$-module. If $M$ is a sequentially Cohen-Macaulay $A$-module, then $R$ is a sequentially Cohen-Macaulay local ring.
\end{cor}

\begin{proof}\label{cor2}
We write $M = R \oplus N$ where $N$ is an $R$-submodule of $M$. Since $M$ is a sequentially Cohen-Macaulay $A$-module,  by Theorem \ref{thm2} it is a sequentially Cohen-Macaulay $R$-module as well, so that  by Proposition \ref{prop1}, $R$ is a sequentially Cohen-Macaulay local ring.
\end{proof}

\begin{cor}\label{cor3}
Suppose that $R$ is a direct summand of $A$ as an $R$-module. If $A$ is a sequentially Cohen-Macaulay local ring, then $R$ is a sequentially Cohen-Macaulay local ring.
\end{cor}

We consider the invariant subring $R = A^G$. 

\begin{cor}\label{cor4}
Let $A$ be a Noetherian local ring, $G$ a finite subgroup of $\operatorname{Aut}A$. Suppose that the order of $G$ is invertible in $A$.  If $A$ is a sequentially Cohen-Macaulay local ring, then the invariant subring $R=A^G$ of $A$ is a sequentially Cohen-Macaulay local ring.
\end{cor}

\begin{proof}
Since the order of $G$ is invertible in $A$, $A$ is a module-finite extension of $R= A^G$ such that $R$ is a direct summand of $A$ (see \cite{BR} and reduce to the case where $A$ is a reduced ring). Hence the assertion follows from Corollary \ref{cor3}.
\end{proof}

\begin{rem} In the setting of Corollary \ref{cor4}, let $\{D_i\}_{0 \le i \le \ell}$ be the dimension filtration of $A$. Then every $D_i$ is a $G$-stable ideal of $A$  (compare with Theorem \ref{thm2} (1)) and the dimension filtration of $R$ is given by a refinement of $\{D_i^G\}_{0 \le i \le \ell}$. 
\end{rem}

The goal of this section is the following. 

\begin{cor}\label{cor5} 
Let $R$ be a Noetherian local ring, $M\neq (0)$ a finitely generated $R$-module.  We put $A = R\ltimes M$ the idealization of $M$ over $R$. Then the following conditions are equivalent.
\begin{enumerate}
\item[$(1)$] $A = R \ltimes M$ is a sequentially Cohen-Macaulay local ring.
\item[$(2)$] $A=R\ltimes M$ is a sequentially Cohen-Macaulay $R$-module. 
\item[$(3)$] $R$ is a sequentially Cohen-Macaulay local ring and $M$ is a sequentially Cohen-Macaulay $R$-module. 
\end{enumerate}
\end{cor}

%%%%%%%%%%%%%%%%%%%%%%%%%%%%%%%%%%%%%%%%%%%%%%%%%%%%%%%%%%%%%%%%%%%%%%%%%%%%%%%%%%%%%%%%%%%%%%%%%%%%%%%%%%%%%%%%%%%%%%%%%%%%%%%%%%%%%%%%%%%%%%%%%%%%%%%%%%%%%%%%%%%%%%%%%%%%%%%%%%%%%%%%%%%%%%%%%%%%%%%%%%%%%

\vspace{0.5cm}

\begin{ac}
The authors would like to thank Professor Shiro Goto for his valuable advice and comments.
\end{ac}

%%%%%%%%%%%%%%%%%%%%%%%%%%%%%%%%%%%%%%%%%%%%%%%%%%%%%%%%%%%%%
%\addcontentsline{toc}{section}{references}

%\begin{thebibliography}{99}

%\bibitem{BH} W. Bruns and J. Herzog, Cohen-Macaulay rings, Cambridge studies in advanced mathematics {\bf 39}, Cambridge University Press, 1993.


\begin{thebibliography}{GHS}

\bibitem[BR]{BR}
{\sc J. W. Brewer and E. A. Rutter}, \textit{Must $R$ be Noetherian if $R^G$ is Noetherian}, Comm. Algebra, {\bf 5} (1977), 969-979.

%\bibitem[CC]{CC}
%{\sc N. T. Cuong and D. T. Cuong}, {\em On sequentially Cohen-Macaulay modules,} Kodai Math. J., {\bf (30)}  (2007), 409-428.

\bibitem[CGT]{CGT} 
{\sc N. T. Cuong, S. Goto and H. L. Truong}, {\em The equality $I^2=\q I$ in sequentially Cohen-Macaulay rings}, J. Algebra, {\bf (379)}  (2013), 50-79.


\bibitem[CN] {CN}  N. T. Cuong and L. T. Nhan {\em Pseudo Cohen-Macaulay and pseudo generalized Cohen-Macaulay modules},   J. Algebra, {\bf (267)}  (2003), 156-177.

%\bibitem[CT] {CT}  N. T. Cuong and H. L. Tuong {\em Parametric Decomposition of Powers of Parameter Ideals and Sequentially Cohen-Macaulay Modules},   J. AA, {\bf }  (2009), .

%\bibitem[F]{F} 
%{\sc G. Faltings}, \textit{\"{U}ber die Annulatoren lokaler Kohomologiegruppen}, Archiv der Math., {\bf 30} (1978), 473--476.


\bibitem[GHS]{GHS}
{\sc S. Goto, Y. Horiuchi and H. Sakurai}, \textit{Sequentially Cohen-Macaulayness versus parametric decomposition of powers of parameter ideals}, J. Comm. Algebra, {\bf 2} (2010), 37--54.

%\bibitem[GN]{GN} 
%{\sc S. Goto and K. Nishida}, \textit{The Cohen-Macaulay and Gorenstein properties of Rees algebras associated to fltrations}, Mem. Amer. Math. Soc., {\bf 110} (1994).

%\bibitem[GS]{GS}
%{\sc S. Goto and Y. Shimoda}, \textit{On the Rees algebra of Cohen-Macaulay local rings}, Commutative Algebra, Lecture Note in Pure and Applied Mathematics, Marcel Dekker Inc, {\bf 68} (1982), 201-231.


%\bibitem[GW]{GW} 
%{\sc S. Goto and K. Watanabe}, {\it On graded rings, I}, J. Math. Soc. Japan, {\bf 30} (1978), 179--213.

%\bibitem[HS]{HS} 
%{\sc C. Huneke, I. Swanson}, \textit{Integral Closure of Ideals,  Rings and Modules}, Cambridge University Press, London Mathematical Society Lecture Note Series 336.

\bibitem[Sch]{Sch}
{\sc P. Schenzel}, \textit{On the dimension filtration and Cohen-Macaulay filtered modules},  in: Proc. of the Ferrara Meeting in honour of Mario Fiorentini, University of Antwerp, Wilrijk, Belgium, (1998) pp. 245-264.

%\bibitem[St]{St} 
%{\sc R. P. Stanley}, Combinatorics and commutative algebra, Second Edition, Birkh{\" a}user, Boston, 1996.


%\bibitem[TI]{TI}  
%{\sc N. Trung and S. Ikeda}, \textit{When is the Rees algebra Cohen-Macaulay?}, Comm. Algebra {\bf 17} (1989), 2893-2922.


%\bibitem[V]{V}  
%{\sc D. Q. Viet}, \textit{A note on the Cohen-Macaulayness of Rees Algebra of filtrations}, Comm. Algebra {\bf 21} (1993), 221-229.




\end{thebibliography}
\end{document}